\documentclass[10pt,letterpaper,dvips]{article}

 \usepackage{amssymb,latexsym,amsmath,amsthm}

\newtheorem*{thm*}{Theorem} 

\newtheorem{thm}{Theorem}

\newtheorem{dfn}{Definition}

\newtheorem{lemma}{Lemma}

\newtheorem*{lemma*}{Lemma}

\newtheorem{remark}{Remark}

\newtheorem{cor}{Corollary}

\begin{document}

\def\d{ \partial_{x_j} } 
\def\Na{{\mathbb{N}}}

\def\Z{{\mathbb{Z}}}

\def\IR{{\mathbb{R}}}

\newcommand{\E}[0]{ \varepsilon}

\newcommand{\la}[0]{ \lambda}

\newcommand{\s}[0]{ \mathcal{S}}

\newcommand{\AO}[1]{\| #1 \| }

\newcommand{\BO}[2]{ \left( #1 , #2 \right) }

\newcommand{\CO}[2]{ \left\langle #1 , #2 \right\rangle} 

\newcommand{\R}[0]{ \IR\cup \{\infty \} } 

\newcommand{\co}[1]{ #1^{\prime}} 

\newcommand{\p}[0]{ p^{\prime}} 

\newcommand{\m}[1]{   \mathcal{ #1 }} 

\newcommand{ \W}[0]{ \mathcal{W}}

\newcommand{ \A}[1]{ \left\| #1 \right\|_H }

\newcommand{\B}[2]{ \left( #1 , #2 \right)_H }

\newcommand{\C}[2]{ \left\langle #1 , #2 \right\rangle_{  H^* , H } }

 \newcommand{\HON}[1]{ \| #1 \|_{ H^1} }

\newcommand{ \Om }{ \Omega}

\newcommand{ \pOm}{\partial \Omega}

\newcommand{\D}{ \mathcal{D} \left( \Omega \right)}

\newcommand{\DP}{ \mathcal{D}^{\prime} \left( \Omega \right)  }

\newcommand{\DPP}[2]{   \left\langle #1 , #2 \right\rangle_{  \mathcal{D}^{\prime}, \mathcal{D} }}

\newcommand{\PHH}[2]{    \left\langle #1 , #2 \right\rangle_{    \left(H^1 \right)^*  ,  H^1   }    }

\newcommand{\PHO}[2]{  \left\langle #1 , #2 \right\rangle_{  H^{-1}  , H_0^1  }} 

 \newcommand{\HO}{ H^1 \left( \Omega \right)}

\newcommand{\HOO}{ H_0^1 \left( \Omega \right) }

\newcommand{\CC}{C_c^\infty\left(\Omega \right) }

\newcommand{\N}[1]{ \left\| #1\right\|_{ H_0^1  }  }

\newcommand{\IN}[2]{ \left(#1,#2\right)_{  H_0^1} }

\newcommand{\INI}[2]{ \left( #1 ,#2 \right)_ { H^1}} 

\newcommand{\HH}{   H^1 \left( \Omega \right)^* } 

\newcommand{\HL}{ H^{-1} \left( \Omega \right) }

\newcommand{\HS}[1]{ \| #1 \|_{H^*}}

\newcommand{\HSI}[2]{ \left( #1 , #2 \right)_{ H^*}}

\newcommand{\WO}{ W_0^{1,p}} 
\newcommand{\w}[1]{ \| #1 \|_{W_0^{1,p}}}  

\newcommand{\ww}{(W_0^{1,p})^*}   

\newcommand{\Ov}{ \overline{\Omega}}

\title{Regularity of semi-stable solutions to fourth order nonlinear eigenvalue problems on general domains}
\author{Craig Cowan \\
{\it\small Department of Mathematical Sciences}\\
{\it\small University of Alabama in Huntsville}\\
{\it\small 258A Shelby Center}\\
\it\small Huntsville, AL 35899 \\
\it\small cowan@math.uah.edu \\
{\it\small ctcowan@stanford.edu}\vspace{1mm}\and
Nassif Ghoussoub\thanks{Partially supported by a grant from the Natural Sciences and Engineering Research Council of Canada}\hspace{2mm}\\
{\it\small Department of Mathematics}\\
{\it\small University of British Columbia,}\\
{\it\small Vancouver BC Canada V6T 1Z2}\\
{\it\small  nassif@math.ubc.ca}\\
}

\maketitle


\vspace{3mm}

\begin{abstract} We examine the fourth order problem $\Delta^2 u = \lambda f(u) $ in $ \Omega$ with $ \Delta u = u =0 $ on $ \pOm$, 
where $ \lambda > 0$ is a parameter, $ \Omega$ is a bounded domain in $ \IR^N$ and where $f$ is one of the following nonlinearities: $ f(u)=e^u$, $ f(u)=(1+u)^p $ or $ f(u)= \frac{1}{(1-u)^p}$ where $ p>1$.   We show the regularity of all semi-stable solutions and hence of the extremal solutions, provided \[  N < 2 + 4 \sqrt{2} + 4 \sqrt{ 2 - \sqrt{2}} \approx 10.718 \mbox{\; \; when $ f(u)=e^u$,}\]  and
\[ \frac{N}{4} < \frac{p}{p-1} + \frac{p+1}{p-1} \left( \sqrt{ \frac{2p}{p+1}} + \sqrt{  \frac{2p}{p+1} - \sqrt{  \frac{2p}{p+1}}} - \frac{1}{2} \right)\] when $ f(u)=(u+1)^p$.   New results are also obtained in the case where $ f(u)=(1-u)^{-p}$.  These are substantial improvements to various results on critical dimensions obtained recently by various authors.    We view the equation as a system and then derive a new stability inequality, valid for minimal solutions, which allows a method of proof which is reminiscent of the second order case. 

\end{abstract}
\noindent
{\it \footnotesize 2010 Mathematics Subject Classification}. {\scriptsize }

\section{Introduction} 
We examine the problem 
\begin{equation*} 
(N)_\lambda \qquad \left\{ 
\begin{array}{ll}
\Delta^2 u = \lambda f(u) & \hbox{in }\Omega \\
u =\Delta u = 0 &\hbox{on } \pOm, 
\end{array}
\right.
\end{equation*} 
where $ \lambda \ge 0$ is a parameter, $ \Omega$ is a bounded domain in $ \IR^N$, $N\geq 2$, and where $f$ is one of the following nonlinearities: $ f(u)=e^u$, $ f(u)=(1+u)^p $ or $ f(u)= \frac{1}{(1-u)^p}$ where $ p>1$.     We are interested in obtaining regularity results concerning the extremal solution $ u^*$ associated with $(N)_{\lambda^*}$.   

 The nonlinearities we examine naturally fit into the following two classes: \\

(R): \qquad $f$ is smooth, increasing, convex on  $\IR$ with $ f(0)=1$ and $ f$ is superlinear at $ \infty$ (i.e. $ \displaystyle \lim_{u \rightarrow \infty} \frac{f(u)}{u}=\infty$);\\

(S): \qquad $f$ is smooth, increasing, convex on $[0, 1)$ with $ f(0)=1$ and $\displaystyle  \lim_{u \nearrow 1} f(u)=+\infty$.  \\  

     Before we discuss some known results concerning the problem $(N)_\lambda$ we recall various facts concerning second order version of the above problem.

\subsection{The second order case} 
For a nonlinearity $ f$ of type (R) or (S), the following  second order analog of $(N)_\lambda$ with Dirichlet boundary conditions
\begin{equation*}
(Q)_\lambda \qquad  \left\{ 
\begin{array}{ll}
-\Delta u =\lambda f(u) &\hbox{in }\Omega \\
u =0 &\hbox{on } \pOm
\end{array}
\right.
\end{equation*} 
is by now quite well understood whenever $ \Omega$ is a bounded smooth domain in $ \IR^N$. See, for instance, \cite{BV,Cabre,CC,EGG,GG,Martel,MP,Nedev}. We now list the  properties one comes to expect when studying $(Q)_\lambda$.  

\begin{itemize} \item  There exists a finite positive critical parameter $ \lambda^*$ such that for all $ 0< \lambda < \lambda^*$ there exists a \textbf{minimal solution} $ u_\lambda$ of $ (Q)_\lambda$.   By minimal solution, we mean here that if $ v$ is another solution of $ (Q)_\lambda$ then $v \ge u_\lambda$ a.e. in $ \Omega$.  

\item For each $ 0< \lambda < \lambda^*$ the minimal solution $ u_\lambda$ is \textbf{semi-stable} in the sense that 
\[ \int_\Omega \lambda f'(u_\lambda) \psi^2 dx \le \int_\Omega | \nabla \psi|^2 dx, \qquad \forall \psi \in H_0^1(\Omega),\]  
and is unique among all the weak semi-stable solutions. 
 
\item The map $ \lambda \mapsto u_\lambda(x)$ is increasing on $ (0,\lambda^*)$ for each $ x \in \Omega$.    This allows one to define $ u^*(x):= \lim_{\lambda \nearrow \lambda^*} u_\lambda(x)$, the so-called {\bf extremal solution}, which can be shown to be a weak solution of $ (Q)_{\lambda^*}$.    In addition one can show that $ u^*$ is the unique weak solution of $(Q)_{\lambda^*}$. See \cite{Martel}. 
\item There are no solutions of $ (Q)_{\lambda}$ (even in a very weak sense) for $ \lambda > \lambda^*$.   

\end{itemize} 

A question which has attracted a lot of attention is whether the extremal function $ u^*$ is a classical solution of $(Q)_{\lambda^*}$. This is of interest since one can then apply the results from \cite{CR}  to start a second branch of solutions emanating from $(\lambda^*, u^*)$.    Note that in the case where $ f$ satisfies (R)  (resp. (S)) it is sufficient, in view of standard elliptic regularity theory, to show that $ u^*$ is bounded  (resp.  $\sup_\Omega u^* <1$). 

The answer typically depends on the nonlinearity $f$, the dimension $N$ and the geometry of the domain $ \Omega$.  We now list some known results.
 
 \begin{itemize} \item  \cite{CR} Suppose $ f(u)=^u$.  If  $ N <10$ then $ u^*$ is bounded.  For $ N \ge 10$ and $ \Omega$ the unit ball $ u^*(x)=-2 \log(|x|)$.

 \item \cite{CC} Suppose $ f$ satisfies (R) but without the convexity assumption and $ \Omega$ is the unit ball.   Then $ u^*$ is bounded for $ N <10$.  In view of the above result this is optimal.

 \item On general domains, and if $f$ satisfies (R), then $ u^*$ is bounded for $ N \le 3$ \cite{Nedev}.   Recently this has been improved to $ N \le 4$ provided the domain is convex (again one can drop the convexity assumption on $f$), see \cite{Cabre}.  
 
 \item  \cite{GG} Suppose $ f(u)=(1-u)^{-2}$. Then $ \sup_\Omega u^*<1$ for $ N \le 7$ and in the case of the unit ball $ u^*(x)= 1 - |x|^\frac{2}{3}$ for $ N \ge 8$. 
   
  \end{itemize}

In the previous list, we have not considered the nonlinearity $f(u)=(u+1)^p$, $p>1$, 
since many of the formula's become a bit cumbersome.

\subsection{The fourth order case}      
There are two obvious fourth order extensions of $ (Q)_{\lambda}$ namely the problem $(N)_\lambda$ mentioned above, and its Dirichlet counterpart 
 \begin{equation*}
(D)_\lambda \qquad  \left\{ 
\begin{array}{ll}
\Delta^2 u =\lambda f(u) &\hbox{in }  \Omega  \\
u = \partial_\nu u = 0 & \hbox{on }  \pOm, 
\end{array}
\right.
\end{equation*}   where $ \partial_\nu$ denote the normal derivative on $ \pOm$.    The problem $(Q)_{\lambda}$ is heavily dependent on the  maximum principle and hence this poses a major hurdle in the study of $(D)_\lambda$ since for general domains there is no maximum principle for $ \Delta^2$ with Dirichlet boundary conditions.  But if we restrict our attention to the unit ball then one does have a weak maximum principle \cite{BOG}.      The problem $(D)_\lambda$ was studied in \cite{AGGM} and various results were obtained,  but results concerning the boundedness of the extremal solution (for supercritical nonlinearities) were missing.   

 The first (truly supercritical) results concerning the boundedness of the extremal solution in a fourth order problem are due to \cite{DDGM} where they examined the problem $(D)_\lambda$ on the unit ball in $ \IR^N$ with $ f(t)=e^t$.   They showed that the extremal solution $ u^*$ is bounded if and only if $ N \le 12$.     Their approach is heavily dependent on the fact that $\Omega$ is the unit ball. 
Even in this situation there are two main hurdles.  The first is that the standard energy estimate approach, which was so successful in the second order case, does not appear to work in the fourth order case.   The second is the fact that it is quite hard to construct explicit solutions of $(D)_\lambda$ on the unit ball that satisfy both boundary conditions, which is needed to show that the extremal solution is unbounded for $ N \ge 13$.   So what one does is to find an explicit singular, semi-stable solution which satisfies the first boundary condition, and then to perturb it  enough to satisfy the second boundary condition but not too much so as to lose the semi-stability. Davila et al. \cite{DDGM} succeeded in doing so for $ N \ge 32$, but they were forced to use a computer assisted proof to show that the extremal solution is unbounded for the intermediate dimensions $ 13 \le N \le 31$.  Using various improved Hardy-Rellich inequalities from \cite{GM} the need for the computer assisted proof was removed in \cite{Moradifam}. The case where $ f(t)=(1-t)^{-2}$ was settled at the same time in \cite{CEG}, where we used methods developed in \cite{DDGM} to show that the extremal solution associated with $(D)_\lambda$ is a classical solution if and only if $ N \le 8$.

We now examine the problem $(N)_\lambda$.  The Navier boundary conditions allow for a maximum principle and hence many of the results that hold for $(Q)_\lambda$ also hold for $(N)_\lambda$.  We now list some basic properties and give a few definitions.  See \cite{BG, CDG} for more details.

\begin{dfn}   Given a smooth solution $u$ of $ (N)_\lambda$,  we say that $u$ is a semi-stable solution of $(N)_\lambda$ if 
\begin{equation} \label{class_stable}
\int  \lambda f'(u) \psi^2 dx \le \int  (\Delta \psi)^2 dx, \qquad \forall \psi \in H^2(\Omega) \cap H_0^1(\Omega).
\end{equation} 
\end{dfn}   

\begin{dfn}  We say a smooth solution  $ u$ of $(N)_\lambda$  is minimal  provided $ u \le v$ a.e. in $ \Omega$ for any solution $ v$ of $(N)_\lambda$. 
\end{dfn} 

We define the {\bf extremal parameter} $ \lambda^*$ as 
\[ \lambda^*:= \sup \left\{ 0 < \lambda: \mbox{there exists a smooth solution of $(N)_\lambda$} \right\}.\]

It is known, see \cite{BG,CDG,GW}, that: 
\begin{enumerate} \item  $0 < \lambda^* < \infty$.  

\item For each $ 0 < \lambda < \lambda^*$ there exists a smooth minimal solution $ u_\lambda$ of $(N)_\lambda$.  Moreover the minimal solution $ u_\lambda$ is semi-stable and is unique among the 
semi-stable solutions.  

\item For each $ x \in \Omega$, $ \lambda \mapsto u_\lambda(x)$ is strictly increasing on $(0, \lambda^*)$, and it therefore makes sense to define  $ u^*(x):= \lim_{\lambda \nearrow \lambda^*} u_\lambda(x), $   which we call the \textbf{extremal solution}.
\item  There are no smooth solutions for $ \lambda > \lambda^*$.

\end{enumerate} 

  It is standard to show that $u^*$  is a weak solution of $(N)_{\lambda^*}$ in a suitable sense that we shall not define here since it will not be needed in the sequel.  

We now examine some known results for the regularity of the extremal solution.  If the domain is the unit ball, then again one can use the methods of \cite{DDGM} and \cite{CEG} to obtain optimal results in the case of  $ f(t)=(1-t)^{-2}$ (see for instance \cite{memsbook} and \cite{MoradifamMEMS}).        For general domains the regularity of the extremal solution was limited to subcritical and critical nonlinearities, see \cite{BG,GW,memsbook}.   This was improved in \cite{craig1}
to include a range of supercritical nonlinearites.  We now list the results from \cite{craig1}.  Suppose $ \Omega$ a bounded domain in $ \IR^N$.  

\begin{itemize}   \item Suppose $ f(t)=e^t$ or $ f(t)=(t+1)^p$ for any $ 1<p$.  Then $ u^*$ is bounded for $ N \le 8$. 

\item Suppose $ f(t)=(t+1)^p$ where $ 1<p$ and $ N < \frac{8p}{p-1}$.  Then $ u^*$ is bounded. 

\item Suppose $f$ satisfies (R) and $ N \le 5$. Then $ u^*$ is bounded. 

\item Suppose $ f(t)=(1-t)^{-p}$ where $ 1 < p \neq 3$ and $ N \le \frac{8p}{p+1}$.  Then $ \sup_{\Omega}u^* <1$. 

 \end{itemize}  Various other results were obtained provided $ f,f', f''$ satisfy various constraints.   These are major improvements over the subcritical and critical results but these are still far from being optimal after one considers the results when $ \Omega$ is a ball.      

  We give a very brief outline of the proof used to prove the results from \cite{craig1}.  Firstly assume that $u$ is a smooth minimal solution of $(N)_\lambda$.   \\
  Step 1.  Test (\ref{class_stable}) on $ \psi:=\Delta u$ and integrate by parts.  One then sees the highest order terms cancel and one is left with 
  \[ \int f''(u) (-\Delta u) | \nabla u|^2 \le \lambda \int f(u).\]  
  
  The second key step is to obtain a pointwise lower estimate on $ -\Delta u$ and then one puts this back into the inequality from step 1 to obtain a useable estimate.  The next theorem gives the pointwise lower estimate.  
  We remark that this result is inspired by \cite{soup}.

\begin{thm*} A \cite{craig1} \label{max}. Suppose $ u$ is a solution of $ (N)_\lambda$ and $g$ is a smooth function defined on the range of $u$ with $ f(t) \ge g(t) g'(t)$ and $ g(t), g'(t), g''(t) \ge 0$ on the range of $ u$ with $ g(0)=0$.   Then 
\begin{equation}-\Delta u \ge \sqrt{\lambda} g(u) \qquad \mbox{ in $ \Omega$}.
\end{equation}
\end{thm*}

\begin{proof}  Define $ w:=-\Delta u - \sqrt  \lambda g(u)$ and so $ w =0$ on $ \pOm$ and a computation shows that 
\[ -\Delta w + \sqrt  \lambda g'(u) w = \lambda [f(u) -g(u) g'(u)] + \sqrt  \lambda g''(u) | \nabla u|^2 \qquad \mbox{ in $ \Omega$}.\]   The assumptions on $ g$ allow one to apply the maximum principle and obtain that $ w \ge 0$ in $ \Omega$.  \\
\end{proof}

Note that $ g(u):= \sqrt{2} \left( \int_0^u (f(t)-1) dt \right)^\frac{1}{2},$ satisfies the hypothesis of Theorem A and this was used in \cite{craig1} to obtain results regarding arbitrary nonlinearities $f$.  In the current work we don't examine arbitrary nonlinearities and rather than use this general formula for the pointwise bounds we prefer some slightly different choices of $g$.     

There have been two works \cite{Wei_ye,Wei} that improve on some of the results from \cite{craig1}.   In both works the main interest was in the existence of nontrivial stable   solutions of $ \Delta^2 u= u^p$ in $ \IR^N$.  As a byproduct they obtained improved regularity results concerning $(N)_{\lambda^*}$ in the case where $ f(u)=(u+1)^p$ which we now state.

\begin{itemize} \item \cite{Wei_ye}   For each $ 9 \le N \le 19$ there is some $ \E_N>0$ such that the extremal solution associated with $ (N)_{\lambda^*}$ is bounded provided  $ 1 <p< \frac{N}{N-8} +\E_N$.     One does not have estimates on $ \E_N$.

\item  \cite{Wei}  For each $N \ge 20$ the extremal solution associated with $(N)_{\lambda^*}$ is bounded provided 
\[ 1 < p < 1 + \frac{8 p_N^*}{N-4}.\]    Here $ p_N^*$ stands for the smallest real root which greater than $ \frac{N-4}{N-8}$ of the following equation: 

\begin{eqnarray*}
512(2-N) x^6 + 4(N^3-60N^2+670N-1344)x^5 && \\
-2(13N^3-424N^2+3064N -5408)x^4 
+ 2(27N^3-572N^2 +3264N-5440)x^3 && \\ -(49N^3 -772N^2+3776 N - 5888)x^2 
+4(5 N^3 -66 N^2 +288N -416)x  && \\ -3(N^3-12N^2+48N-64)=0
\end{eqnarray*}

\end{itemize}

\section{New approach}

The main idea of this work is to view $(N)_\lambda$ as the system 

\begin{equation*} 
(N')_\lambda \qquad  \left\{ 
\begin{array}{ll}
-\Delta u = v & \hbox{in }\Omega \\
-\Delta v = \lambda f(u) &\hbox{on } \pOm, \\
u = v = 0 & \hbox{on } \pOm, \\
\end{array}
\right.
\end{equation*}   and then to use the results from \cite{Mont} to derive a new stability inequality which 
gives improved estimates on minimal solutions.   
We mention that there have been a few works which use the results from \cite{Mont} to obtain a stability like inequality for minimal solutions of  various systems.  In \cite{craig0} a stability like inequality was obtained, using \cite{Mont}, which was useful for showing regularity of the extremal solutions associated with  $ -\Delta u = \lambda e^u, \; \; -\Delta v = \gamma e^u$. 
   This new approach is motivated by portions of \cite{cowan_fazly} and \cite{fg}.

The result from \cite{Mont} we use is: 
\begin{thm*} B \cite{Mont}.  Suppose $ (u,v)$ is a smooth minimal solution of $ -\Delta u = \gamma G(u,v), \; \; -\Delta v = \lambda F(u,v)$ in $ \Omega$ with $ u=v=0 $ on $ \pOm$.  Here $ F(u,v),G(u,v)$ are positive nonlinearities which are increasing in $ u$ and $ v$.  Then there is some $ \eta \ge 0$ and $ 0 < \phi,\psi \in H_0^1(\Omega)$ such that 
\[ -\Delta \phi = \gamma G_u \phi + \gamma G_v \psi + \eta \phi, \quad -\Delta \psi = \lambda F_u \phi + \lambda F_v \psi + \eta \psi \qquad \mbox{in $\Omega$}.\] 

\end{thm*}

 We now state the general stability inequality we use.  
 
 \begin{lemma} \label{gen_in} Let $(u,v)$ denote a smooth minimal solution of the system from Theorem B.  Then  
 \[ \int \gamma G_u \alpha^2 +  \lambda F_v \beta^2 + 2 \sqrt{\gamma \lambda} \int \sqrt{ F_u G_v} \alpha \beta \le  \int | \nabla \alpha|^2 +  \int | \nabla \beta|^2,\] for all $ \alpha,\beta \in H_0^1(\Omega)$. 
 
 \end{lemma} 

\begin{proof} We rewrite the result from the above theorem as 
\[ \frac{-\Delta \phi}{\phi} \ge \gamma G_u + \frac{\gamma G_v \psi }{\phi}, \qquad \frac{-\Delta \psi}{ \psi} \ge \frac{\lambda F_u \phi}{\psi} + \lambda F_v,\]  and we multiply the first equation by $ \alpha^2 $ and the second equation by $ \beta^2$ where $ \alpha,\beta \in C_c^\infty(\Omega)$ and add the equations and integrate over $ \Omega$ to obtain 
\begin{equation} \label{old}
 \int \gamma G_u \alpha^2 + \lambda F_v \beta^2 + \int \gamma G_v \alpha^2 \frac{\psi}{\phi} + \lambda F_u \beta^2 \frac{\phi}{\psi} \le \int | \nabla \alpha|^2 + \int | \nabla \beta|^2,
\end{equation}  where we used the fact that for any sufficiently regular function $ E>0$ we have 
\[ \int \frac{-\Delta E}{E} w^2 \le \int | \nabla w|^2,\] for all $ w \in C_c^\infty(\Omega)$.

 We now find a lower estimate for the second integral.  For this note that some simple calculus shows that $ a t + \frac{b}{t} \ge 2 \sqrt{ab}$ for all $ t >0$ and any $ a,b>0$.   So we have 
\[ \gamma G_v \alpha^2 \frac{\psi}{\phi} + \lambda F_u \beta^2 \frac{\phi}{\psi} \ge 2 \sqrt{\lambda \gamma} \sqrt{F_u G_v}  \; \alpha \beta.\]  Using this bound one gets the desired result.
\end{proof} 
 We remark that inequalities similar to (\ref{old}) were used in \cite{craig0}  but the new ingredient here is given by the above approximation.   This new approximation allows us to handle systems that the previous method could not be used for.

 One should note that the system $(N')_\lambda$ does not exactly fit into the above framework since $G(u,v)= v$.  One could slightly adjust the proof from \cite{Mont} but we choose not to do this and we take a   slightly different approach to get a similar result.    Let $ u_\lambda$ denote the minimal solution of $ (N)_\lambda$ and set $ v_\lambda:=-\Delta u_\lambda$.  Then $ u_\lambda, v_\lambda$ are increasing in $ \lambda$  and so taking a derivative in $ \lambda$ gives 
 \begin{equation} \label{ppo} 
  -\Delta \phi = \psi, \qquad -\Delta \psi \ge \lambda f'(u_\lambda) \phi,
 \end{equation}  where $ \phi:= \partial_\lambda u_\lambda$ \; $ \psi:= \partial_\lambda v_\lambda$.   Note that $ \phi,\psi \ge 0$ by monotonicity and are positive by the maximum principle. We can now proceed as in Lemma \ref{gen_in} to obtain a stability inequality valid for minimal solutions of $(N)_\lambda$ given by  
 
 \begin{cor}  Let $ 0 < \lambda<\lambda^*$ and $ u_\lambda$ denote the minimal solution of $(N)_\lambda$.  Then 
 \begin{equation} \label{four_sta} 
 \sqrt{\lambda} \int \sqrt{ f'(u_\lambda)} \phi^2 \le \int | \nabla \phi|^2,  
 \end{equation} for all $ \phi \in H_0^1(\Omega)$.
 \end{cor} 

This inequality coupled with pointwise estimates on $ v:=-\Delta u$ from \cite{craig1} (see Theorem A above and Lemma \ref{pointwise})  will be the main tools we use to obtain new energy estimates valid for the extremal solution associated with $(N)_{\lambda^*}$.

\section{Main results}

We now give our main results.  $\Omega$ will always denote a smooth bounded domain in $ \IR^N$.   Our first result deals with the well known examples of $f$ which satisfy (R). 

\begin{thm} \label{expon} Suppose $ f(u)=e^u$ and 
\[ N < 2 + 4 \sqrt{2} + 4 \sqrt{ 2 - \sqrt{2}} \approx 10.718.\]   Then the extremal solution $ u^*$ associated with $(N)_{\lambda^*}$ is bounded.  

\end{thm}

\begin{thm} \label{poly} Suppose $f(u)=(u+1)^p$ where $ 1 <p$ and 
\[ \frac{N}{4} < \frac{p}{p-1} + \frac{p+1}{p-1} \left( \sqrt{ \frac{2p}{p+1}} + \sqrt{  \frac{2p}{p+1} - \sqrt{  \frac{2p}{p+1}}} - \frac{1}{2} \right) =:h(p).\]  Then the extremal solution $ u^*$ associated with $(N)_{\lambda^*}$ is bounded.  
\end{thm}  Note that $ h(p)$ is decreasing in $p$ on $(1,\infty)$ and 
\[  \lim_{p \rightarrow \infty} 4 h(p) = 2 + 4 \sqrt{2} + 4 \sqrt{ 2 - \sqrt{2}} = 10.718...\]

    \begin{thm}  \label{sing} Suppose $f(u)=\frac{1}{(1-u)^p}$ where $ 1 <p \neq3 $ and 
    \[ \frac{N}{4} < \frac{p}{p+1} + \frac{p-1}{p+1} \left(  \sqrt{  \frac{2p}{p-1}} + \sqrt{  \frac{2p}{p-1} - \sqrt{ \frac{2p}{p-1}}} - \frac{1}{2} \right).\]  Then the extremal solution $u^*$ associated with $(N)_{\lambda^*}$ satisfies $ \sup_\Omega u^* <1$. 
    
    \end{thm}

\begin{remark}  \begin{itemize} \item $f(u)=e^u$.    From Theorem \ref{expon} we have $ u^*$ bounded for $ N \le 10$.  Previously $u^*$ was only known to be bounded for $ N \le 8$, \cite{craig1}.   On the ball $u^*$ is bounded if and only if $ N \le 12$, see \cite{DDGM}, and so it is expected $ u^*$ is bounded for $ N \le 12$ on general domains.  

 If one examines the proof from \cite{DDGM} the importance of the radial domain is limited to the following property: there exists some $ x_0 \in \Omega$ such that $ \max_\Omega u_\lambda = u_\lambda(x_0)$ for all $ 0 < \lambda < \lambda^*$.   This property is known to hold for domains $ \Omega$ which are symmetric across each hyperplane $ x_i=0$ and which are convex along the coordinate axis; use the Moving Plane Method or one can use an stability argument.   In any case $ u^*$ is bounded for $ N \le 12$ provided the domain $ \Omega$ satisfies the above conditions.

\item  $f(u)=(u+1)^p$.  From \cite{craig1} we know that the extremal solution is bounded provided  $\frac{N}{4} < \frac{2p}{p-1}$.  One can show that  $ \frac{2p}{p-1} < h(p)$ for all $ p>1$ and so Theorem \ref{poly} is gives an improvement over the results from \cite{craig1}.  Additionally we have $ u^*$ bounded for any $ 1<p$ for $ N \le 10$ where as in \cite{craig1} this only held for $ N \le 8$.  

\item $f(u)=(1-u)^{-p}$.  The most studied case is when $ p=2$ and then $(N)_{\lambda}$ can be seen as a simple model for a Micro-Electro-Mechanical device with pinned boundary conditions.  From \cite{craig1} we know $u^*$ is bounded away from $1$ provided $ N \le 5$.  Theorem \ref{sing} improves this to $N \le 6$.   This falls short of the expected result that $u^*$ is bounded for $ N \le 8$, which holds on the ball, see \cite{MoradifamMEMS}.  We believe this condition $ p \neq 3$ is somewhat artificial and is coming from our proof method;   the case $ p=3$ involves a borderline Sobolev imbedding theorem, see the proof of Lemma C. 

\end{itemize}

\end{remark}

We begin with some pointwise lower bounds on $ -\Delta u$.   The following result follows immediately after considering Theorem A above.   

\begin{lemma}  \label{pointwise}  Suppose $u$ is a smooth solution of $(N)_\lambda$.  

\begin{enumerate} \item Suppose $ f(t) =e^t$.  Then $ -\Delta u \ge \sqrt{2 \lambda } (e^\frac{u}{2}-1)$. 

\item Suppose $ f(t)=(t+1)^p$ where $p>1$.  Then  $ -\Delta u \ge \sqrt{\lambda} \sqrt{ \frac{2}{p+1}} ((u+1)^\frac{p+1}{2}-1)$.

\item Suppose $ f(t)=\frac{1}{(1-t)^p}$ where $ p >1$.  Then $ -\Delta u \ge \sqrt{\lambda} \sqrt{ \frac{2}{p-1}} \left( \frac{1}{(1-u)^\frac{p-1}{2}} -1 \right)$. 

\end{enumerate} 

\end{lemma}

\textbf{Proof of Theorem \ref{expon}.}   To cut down on repetition we first assume $f$ is an arbitrary nonlinearity  satisfying (R) or (S).   Let $u$ denote the smooth minimal solution of $(N)_\lambda$ where $ \frac{\lambda^*}{2} < \lambda < \lambda^*$ and test (\ref{four_sta}) on $ \phi=v^t$ where $ 1<t$ to arrive at 
\[ \sqrt{\lambda} \int \sqrt{f'(u)} v^{2t} \le t^2 \int v^{2t-2} | \nabla v|^2 = \frac{t^2 \lambda}{2t-1} \int f(u)v^{2t-1},\]  where the last equality is obtained by using $(N)_\lambda$.   We now rewrite this as 
\begin{equation} \label{starting_gen} \E \sqrt{\lambda}  \int \sqrt{f'(u)} v^{2t} + \sqrt{\lambda}(1-\E) \int \sqrt{f'(u)} v^{2t-1} v \le \frac{t^2 \lambda}{2t-1} \int f(u) v^{2t-1},
\end{equation} where $ \E>0$ is greater than zero but small.   

We now assume we are in the case given by $ f(z)=e^z$.  We leave the first term as it is but we replace $v$ in the second term using the pointwise lower estimate on $v$ from Lemma \ref{pointwise} and we regroup the resulting inequality to arrive at 
\begin{eqnarray} \label{popopo} 
\left( (1-\E) \sqrt{2} - \frac{t^2}{2t-1} \right) \int e^u v^{2t-1} & + & \frac{\E}{\sqrt{\lambda}} \int e^\frac{u}{2} v^{2t} \nonumber \\
  & \le & (1-\E) \sqrt{2} \int e^\frac{u}{2} v^{2t-1}.
 \end{eqnarray}
We now estimate the right hand side and for this we denote this integral by $I$.  We write the integral $I$ as a sum of integral over various subregions of $\Omega$  
\[ I = \int_{u \ge T} + \int_{ u<T, v \le k} + \int_{ u<T, v >k},\] where $ T>1$ and $ k>1$. Then one easily sees that 
\begin{eqnarray*} 
I & \le & e^\frac{-T}{2} \int_{u \ge T} e^u v^{2t-1} + |\Omega| e^\frac{T}{2} k^{2t-1} \\
 && + \frac{1}{k} \int_{ u<T,v>k} e^\frac{u}{2} v^{2t}
 \end{eqnarray*} and we then replace on the integrals on the right with integrals over the full region $ \Omega$.  Putting this back into (\ref{popopo}) gives

 \begin{eqnarray*} 
 \left( (1-\E) \sqrt{2} - \frac{t^2}{2t-1} - \frac{(1-\E) \sqrt{2}}{e^\frac{T}{2}} \right) \int e^u v^{2t-1} + && \nonumber \\
 + \left( \frac{\E}{ \sqrt{\lambda}} - \frac{(1-\E) \sqrt{2}}{k} \right) \int e^\frac{u}{2} v^{2t} && \nonumber \\
  \quad \le (1-\E) \sqrt{2} | \Omega| e^\frac{T}{2} k^{2t-1}. 
  \end{eqnarray*}
   We now assume that $ 1 <t< t_0:= \sqrt{2} + \sqrt{ 2 - \sqrt{2}}$ and so $ \sqrt{2} - \frac{t^2}{2t-1}>0$. We now pick $ \E>0$ but sufficiently small such that $ (1-\E) \sqrt{2} - \frac{t^2}{2t-1}$ is still positive.  We then pick $ T$ and $k$ sufficiently big such that both coefficients multiplying the integrals are positive.  We then see that for all $ 1 <t <t_0$  there is some $ C_t <\infty$ such that 
   \[ \int e^u v^{2t-1} \le C_t,\] for all $ \frac{\lambda^*}{2} < \lambda < \lambda^*$.  We now use the lower bound for $v$ again to see that  for all $ 1 <t<t_0$ there is some $ \tilde{C}_t < \infty$ such that 
   \[ \int e^{(t+ \frac{1}{2})u} \le \tilde{C}_t,\] for $  \frac{\lambda^*}{2} < \lambda < \lambda^*$.  Now using elliptic regularity and the Sobolev imbedding theorem we see that we will get a uniform $L^\infty$ bound on the smooth minimal solutions (and hence also on $ u^*$) provided $ t_0 + \frac{1}{2} > \frac{N}{4}$.  We rewrite this as 
   \[ N < 2 + 4 \sqrt{2} + 4 \sqrt{ 2 - \sqrt{2}} = 10.718...\]

\hfill $ \Box$

\textbf{Proof of Theorem \ref{poly}.}  Let $ u$ denote the  smooth minimal solution of $(N)_\lambda$ where $ \frac{\lambda^*}{2} < \lambda < \lambda^*$.  Our starting point is the general formula (\ref{starting_gen}).  Plugging in $ f(z)=(z+1)^p$ one obtains, after using the pointwise lower estimate for $v$ in the second integral of (\ref{starting_gen}) and regrouping,

\begin{eqnarray*}
\left( (1-\E) \sqrt{ \frac{2p}{p+1}} - \frac{t^2}{t-1} \right) \int (u+1)^p v^{2t-1} & +&  \E \frac{\sqrt{p}}{\sqrt{\lambda}} \int (u+1)^\frac{p-1}{2} v^{2t}\\
     & \le & (1-\E) \sqrt{ \frac{2p}{p+1}} \int (u+1)^\frac{p-1}{2} v^{2t-1}. 
\end{eqnarray*} We now estimate the integral on the right, which we denote by $I$, using the same approach as in the proof of Theorem \ref{expon}.  So let $ 1<T,k$ and note 
\begin{eqnarray*}
I&=& \int (u+1)^\frac{p-1}{2} v^{2t-1} \\
&=& \int_{u+1 \ge T} + \int_{u+1 <T, v <k} + \int_{ u+1 <T,v\ge k} \\
& \le & \frac{1}{T^\frac{p+1}{2}} \int_{u+1 \ge T} (u+1)^p v^{2t-1} + |\Omega| T^\frac{p-1}{2} k^{2t-1} \\
&& + \frac{1}{k} \int_{u+1 <T, v \ge k} (u+1)^\frac{p-1}{2} v^{2t}\\
& \le & \frac{1}{T^\frac{p+1}{2}} \int  (u+1)^p v^{2t-1} + |\Omega| T^\frac{p-1}{2} k^{2t-1}  \\
&& + \frac{1}{k} \int (u+1)^\frac{p-1}{2} v^{2t}.
\end{eqnarray*}  Putting this estimate back into the previous inequality and collecting terms gives 

\begin{eqnarray} \label{dd}
\left( (1-\E) \sqrt{ \frac{2p}{p+1}} - \frac{t^2}{2t-1} - \frac{(1-\E)}{T^\frac{p+1}{2}} \sqrt{ \frac{2p}{p+1}} \right) \int (u+1)^p v^{2t-1} + && \nonumber \\
 + \left( \frac{\E \sqrt{p}}{ \sqrt{\lambda}} - \frac{(1-\E)}{k} \sqrt{ \frac{ 2p}{p+1}} \right) \int (u+1)^\frac{p-1}{2} v^{2t} \nonumber \\ 
 \qquad  \le  (1-\E) \sqrt{ \frac{2p}{p+1}} | \Omega| T^\frac{p-1}{2} k^{2t-1}.
\end{eqnarray}  We now show that for all for all $ 1 <t< t_p$ where 
\[t_p:= \sqrt{ \frac{2p}{p+1}} + \sqrt{  \frac{2p}{p+1}- \sqrt{ \frac{2p}{p+1}}  },\] 
there is some $ C_t< \infty$ such that 
\begin{equation} \label{ztzt}
 \int (u+1)^p v^{2t-1} \le C_t
\end{equation} for all $ \frac{\lambda^*}{2} < \lambda < \lambda^*$ (and hence the same estimate also holds for $u^*$).    To see this note that for $ 1 <t <t_p$  we have $ (1-\E) \sqrt{  \frac{2p}{p+1}} - \frac{t^2}{2t-1} >0$ for sufficiently small $ \E>0$.  We now pick $ T$ and $k$ sufficiently big such that both coefficients in front of the integrals in (\ref{dd}) are  positive.   Using the pointwise lower estimate for $v$ in (\ref{ztzt}) shows that for all $ 1 <t<t_p$ there is some $ \tilde{C}_t$ such that 
\begin{equation} \label{sho}
 \int (u+1)^{ p+ (p+1)( t - \frac{1}{2})} \le \tilde{C}_t,
\end{equation} for all $ \lambda$ as above.   We now rewrite the equation for the extremal solution in the alternate form (which will allow us to avoid a bootstrap argument) as 
\[ \Delta^2 u^* = \lambda^* c(x) u^* + \lambda^* \qquad \Omega,\] with $ u^*=\Delta u^*=0 $ on $ \pOm$.   Here 
\[ 0 \le c(x):=  \frac{ (u^*+1)^p-1}{u^*} \le p (u^*)^{p-1},\]  by convexity.   From \cite{craig1} we have that $ u^* \in W^{4,2}(\Omega)$ and so by elliptic regularity theory we will have $ u^* \in L^\infty$ provided $ c \in L^\theta$ for some $ \theta > \frac{N}{4}$.   Using (\ref{sho}) we see this is equivalent to 
\[ \frac{N}{4} < \frac{p}{p-1} + \frac{p+1}{p-1} (t_p- \frac{1}{2}).\]   Multiplying by $4$ gives the desired result. 

\hfill $ \Box$  

The following result  from \cite{craig1} will be needed to prove Theorem \ref{sing}.

\begin{lemma*} C \cite{craig1}. Let $ u_n$ denote a sequence of smooth solutions of $(N)_{\lambda_n}$ where $ f(u)=(1-u)^{-p}$ and $ 1 < p \neq 3$.  Suppose  there is some $ \alpha > 1$ and $\alpha \ge \frac{(p+1)N}{4p}$ such that $\sup_n \| f(u_n)\|_\alpha < \infty$.
Then $ \sup_n  \| u_n\|_{L^\infty}<1$.  

\end{lemma*}

\begin{proof} We suppose that $ N$ is big enough so that $ \frac{(p+1)N}{4p} >1$,  the lower dimensional cases being similar we omit their details.   If $ f(u_n)$ is bounded in $ L^\frac{(p+1)N}{4p}$, then by elliptic regularity we have $ u_n$ bounded  in $ W^{4,\frac{(p+1)N}{4p}}$.   By the Sobolev imbedding theorem we have  $ u_n$ bounded in the space $ C^{ 4 - \left[ \frac{4p}{p+1} \right] -1,  \left[ \frac{4p}{p+1} \right] +1 - \frac{4p}{p+1}}( \overline{\Omega})$, where $ [\cdot ]$ denotes the floor function.
This naturally breaks into the two cases:  
 \begin{itemize}
\item $1<p<3$ and then  $u_n$ is bounded in $C^{1, \frac{3-p}{p+1}}$ 
\item $p>3$ and $ u_n$ is then bounded in $ C^{0,\frac{4}{p+1}}$.  
\end{itemize}
We now let $ x_n \in \Omega$ be such that $ u_n(x_n) = \max_\Omega u_n$.  We claim that there exists some $ C>0$, independent of $n$, such that 
\[ | u_n(x) - u_n(x_n)| \le C |x-x_n|^\frac{4}{p+1}, \qquad x \in \Omega.\] For the second case this is immediate, while for the first we use the fact that $ \nabla u_n(x_n)=0$ and the fact that there is some $ 0 \le t_n \le 1$ such that 
\begin{eqnarray*}
u_n(x) - u_n(x_n) &=& \nabla u_n( x_n+t_n(x-x_n)) \cdot  (x-x_n) \\
&=& \left(  \nabla u_n( x_n+t_n(x-x_n)) - \nabla u_n(x_n) \right) \cdot  (x-x_n)
\end{eqnarray*} along with the fact that $ \nabla u_n$ is bounded in $ C^{0,\frac{3-p}{p+1}}$ to show the claim.  

To complete the proof, we work towards a contradiction, and assume,  after passing to a subsequence,  that $ u_n(x_n) = 1-\E_n \rightarrow 1$.   By passing to another subsequence, we can assume that $ u_n$ converges in $ C( \overline{\Omega})$ which along with the boundary conditions guarantees that $ x_n \rightarrow x_0 \in \Omega$.    Then one has 
\begin{eqnarray*}
1-u_n(x) & =& 1-u_n(x_n) + u_n(x_n) - u_n(x) \\
&=& \E_n + u_n(x_n) - u_n(x) \\
& \le & \E_n + C |x-x_n|^\frac{4}{p+1}, 
\end{eqnarray*} and so there is some $ C_p>0$ such that  
\[ \left( 1-u_n(x) \right)^\frac{(p+1)N}{4} \le C_p \left( \E_n^\frac{(p+1)N}{4} + |x-x_n|^N \right).\] From this one sees that 
\[ f(u_n(x))^\frac{(p+1)N}{4p} \ge \frac{C_p^{-1}}{\E_n^\frac{(p+1)N}{4} + |x-x_n|^N}:=h_n(x).\]   But since $ x_n \rightarrow x_0 \in \Omega$ and $ \E_n \rightarrow 0$,  ones sees that $ \int_\Omega h_n(x) dx \rightarrow \infty$  which contradicts the integrability condition on $ f(u_n)$.  Hence we must have $ \sup_n \| u_n \|_{L^\infty} <1$. 

\end{proof}

\textbf{Proof of Theorem \ref{sing}.}    Let $ u$ denote the  smooth minimal solution of $(N)_\lambda$ where $ \frac{\lambda^*}{2} < \lambda < \lambda^*$.  Our starting point is the general formula (\ref{starting_gen}). Our starting point is the general formula (\ref{starting_gen}).  Plugging in $ f(z)=(1-z)^{-p}$ one obtains, after using the pointwise lower estimate for $v$ in the second integral of (\ref{starting_gen}) and regrouping,

\begin{eqnarray} \label{pss}
\left(  (1-\E) \sqrt{ \frac{2p}{p-1} } - \frac{t^2}{2t-1} \right) \int \frac{v^{2t-1}}{(1-u)^p} + && \nonumber \\
\quad + \frac{ \E  \sqrt{p}}{\sqrt{\lambda}} \int \frac{ v^{2t}}{ (1-u)^\frac{p+1}{2}} && \nonumber \\ 
 \le && (1-\E) \sqrt{ \frac{2p}{p-1}} \int \frac{ v^{2t-1}}{(1-u)^\frac{p+1}{2}}.
\end{eqnarray}   We now estimate the integral on the right, which we denote by $I$.  Let $ 0<T<1$ and $ k \ge 1$.
\begin{eqnarray*}  
I &=& \int_{ T \le u \le 1} + \int_{u<T, v<k} + \int_{u<T, v \ge k} \\ 
& \le & (1-T)^\frac{p-1}{2} \int_{T \le u \le 1} \frac{v^{2t-1}}{(1-u)^p} + \int_{u <T,v<k} \frac{v^{2t-1} }{(1-u)^\frac{p+1}{2}}  \\
& & +  \frac{1}{k} \int_{u <T, v \ge k} \frac{v^{2t}}{(1-u)^\frac{p+1}{2}} \\
& \le & (1-T)^\frac{p-1}{2} \int \frac{v^{2t-1}}{(1-u)^p} + \frac{| \Omega| k^{2t-1}}{(1-T)^\frac{p+1}{2}} \\
&& + \frac{1}{k} \int \frac{v^{2t}}{(1-u)^\frac{p+1}{2}}.
\end{eqnarray*}  Define 
\[ \tilde{t}_p:=  \sqrt{ \frac{2p}{p-1}} + \sqrt{   \frac{2p}{p-1} - \sqrt{ \frac{2p}{p-1}}},\] and note that for all $ 1 <t< \tilde{t}_p$ we have 
\[ \sqrt{ \frac{2p}{p-1}} - \frac{t^2}{2t-1}>0.\]   So for $ 1 <t<\tilde{t}_p$ we put this estimate for $I$ back into 
 (\ref{pss}) and then argue that for $ 0 <\E$ sufficiently small, $0<T<1$ sufficiently close to $1$ and for $k \ge 1$ sufficiently large, the resulting inequality has positive coefficients on the left in front the integrals, and hence for all $ 1 <t<\tilde{t}_p$ there is some $C_t<\infty$ such that 
 \begin{equation} \label{shiii}
 \int \frac{v^{2t-1}}{(1-u)^p} \le C_t,
 \end{equation} for all $ \frac{\lambda^*}{2} < \lambda < \lambda^*$. Now using the pointwise bound on $v$ we  see that for all $ 1 <t<\tilde{t}_p$ there is some $ \tilde{C}_t<\infty$ such that 
 \begin{equation} \label{cat}
 \int  \frac{1}{(1-u)^{ p + (p-1)(t - \frac{1}{2})}} \le \tilde{C}_t,
\end{equation}  for all $ \frac{\lambda^*}{2} < \lambda < \lambda^*$.  This estimate along with the previous lemma shows that we have $ u_\lambda$ uniformly bounded away from $1$ provided 
\[ \frac{(p+1)N}{4p} < 1 + \frac{p-1}{p} \left( \tilde{t}_p- \frac{1}{2} \right),\] and this completes the proof after some rearangements.

\hfill $ \Box$

\end{document}